\documentclass[11pt]{article}
\usepackage[margin=1.4in]{geometry}
\usepackage{amssymb,amsmath}              % See geometry.pdf to learn the layout options. There are lots.
\usepackage{graphicx}
\usepackage{amsthm}
\usepackage{amssymb}
\usepackage{color}
\usepackage{epstopdf}
\renewcommand{\epsilon}{\varepsilon}
 \DeclareMathOperator{\spt}{spt}

\def\Xint#1{\mathchoice
{\XXint\displaystyle\textstyle{#1}}%
{\XXint\textstyle\scriptstyle{#1}}%
{\XXint\scriptstyle\scriptscriptstyle{#1}}%
{\XXint\scriptscriptstyle\scriptscriptstyle{#1}}%
\!\int}
\def\XXint#1#2#3{{\setbox0=\hbox{$#1{#2#3}{\int}$ }
\vcenter{\hbox{$#2#3$ }}\kern-.6\wd0}}

\def\dashint{\Xint-}

\def\div{\mathrm{div} }

\newtheorem{theorem}{Theorem}[section]
\newtheorem{lemma}[theorem]{Lemma}
\newtheorem{proposition}[theorem]{Proposition}

\newtheorem{remark}[theorem]{Remark}

\newtheorem{definition}[theorem]{Definition}

\newtheoremstyle{TheoremNum}
        {\topsep}{\topsep}              %%% space between body and thm
        {\itshape}                      %%% Thm body font
        {}                              %%% Indent amount (empty = no indent)
        {\bfseries}                     %%% Thm head font
        {.}                             %%% Punctuation after thm head
        { }                             %%% Space after thm head
        {\thmname{#1}\thmnote{ \bfseries #3}}%%% Thm head spec
    \theoremstyle{TheoremNum}

\title{On the regularity of stationary points of a nonlocal isoperimetric problem}
\author{Dorian Goldman \thanks{dg443@dpmms.cam.ac.uk, DPMMS University of Cambridge, Cambridge (UK)} \and Alexander Volkmann \thanks{alexander.volkmann@aei.mpg.de, Albert Einstein Institute, Potsdam-Golm}}

\date{}
\begin{document}
\nocite{Volkmann:2010}
\maketitle

\begin{abstract}
In this article we establish $C^{3,\alpha}$-regularity of the reduced boundary of stationary
points of a nonlocal isoperimetric problem in a domain $\Omega \subset \mathbb{R}^n$. In particular, stationary points satisfy the corresponding Euler-Lagrange equation classically on the reduced boundary. Moreover, we show that the singular set has zero $(n-1)$-dimensional Hausdorff measure. This complements the results in \cite{Choksi-Sternberg:2007} in which the Euler-Lagrange equation was derived under the assumption of $C^2$-regularity of the topological boundary and the results in \cite{Sternberg-Topaloglu:2011} in which the authors assume local minimality. 
In case $\Omega$ has non-empty boundary, we show that stationary points meet the boundary of $\Omega$ orthogonally in a weak sense, unless they have positive distance to it. 
%Finally, we prove regularity of local minimizers near the boundary of $\Omega$.
\end{abstract}

\section{Introduction}
The main goal of this work is to establish $C^{3,\alpha}$-regularity of the reduced boundary of stationary points of a nonlocal isoperimetric problem, and estimate the size of its singular set. More precisely, we consider the following functional
\begin{equation}\label{nonloceqn1}
\mathcal E_\gamma(E): = P(E,\Omega) + \gamma \int_E\int_E G(x,y)\,dy\,dx+\int_E f(x)\,dx,
\end{equation}
where $\Omega\subset \mathbb R^n$ is a domain (open, connected) of class $C^2$, $E\subset \Omega$ is a bounded set of finite perimeter $P(E,\Omega)$ in $\Omega$, $\gamma  \geq 0$, 
$f\in C_{loc}^{2}(\overline \Omega)$, and $G$ denotes a symmetric ``kernel'' (see below for precise assumptions on $G$).
The reader should think of $G$ as the Green's function of the Laplace operator with Neumann boundary condition in $\Omega$ or the Newtonian potential in case $\Omega= \mathbb R^n$.
 
Physically, the first term in \eqref{nonloceqn1} models surface tension an thus its minimization favors clustering, whereas the second term can be used to model a competing repulsive term. The third term can be used to model additional external forces, cf. \cite{Giusti:1981}.
The functional $\mathcal E_\gamma$ is often referred to as the \emph{sharp-interface Ohta-Kawasaki energy} \cite{Ohta-Kawasaki:1986} in connection with di-block copolymer melts. 
Minimizers of $\mathcal E_\gamma$ under a volume constraint describe a number of polymer systems \cite{deGennes:1979, Nagaev:1995,Ren-Wei:2000} as well
as many other physical systems \cite{Chen-Khachaturyan:1993,Emery:1993,Glotzer-DiMarzio-Muthukumar:1995,Lundqvist-March:1983,Nagaev:1995} due to the fundamental nature of the Coulombic term. Despite
the abundance of physical systems for which \eqref{nonloceqn1} is applicable, rigorous mathematical analysis for the case $\gamma \neq 0$ is fairly recent.
We refer to the introduction of \cite{Cicalese-Spadaro:2013} for more details and an account of the results about this functional.

Regularity for (local) minimizers of $\mathcal E_\gamma$ under a volume constraint was established by Sternberg and Topaloglu \cite{Sternberg-Topaloglu:2011}. Sternberg and Topaloglu showed that any local minimizer $E$ of $\mathcal E_\gamma$ in a ball $B_\rho(x)$ is a so called \emph{$(K,\varepsilon)$-minimizer} of perimeter in the sense that
\[
P(E,B_\rho(x))\leq P(F,B_\rho(x)) +K\rho^{n-1+\varepsilon} \quad\text{for all $F$ such that $F\Delta E \subset\subset B_\rho(x)$,}
\]
for some $K<\infty$ some $\varepsilon \in (0,1]$.
Standard results (see for example \cite{Giusti:1984,Massari-Miranda:1984,Maggi:2012}) imply that the reduced boundary $\partial^*E\cap B_\rho(x)$ is of class $C^{1,\frac{\varepsilon}{2}}$ and that the singular set $(\partial E\setminus\partial^*E)\cap B_\rho(x)$ has Hausdorff dimension at most $n-8$. Standard elliptic regularity theory then implies higher regularity.

For \emph{stationary} points of $\mathcal E_\gamma$, which are not a priori minimizing in any sense, these methods are no longer available. To this end R\"{o}ger and Tonegawa \cite[Section 7.2]{Roeger-Tonegawa:2008} proved $C^{3,\alpha}$-regularity of the reduced boundary of stationary points of $\mathcal E_\gamma$ that arise as the limit of stationary points of the (diffuse) Ohta-Kawasaki energy with parameter $\varepsilon$ going to zero. They also showed that in this case the singular set has Hausdorff dimension at most $n-1$.

Our main result (Theorem~\ref{thm:regularity}) removes this special assumption. In particular, we do not require any minimality assumptions. As part of our proof we establish a weak measure theoretic form of the Euler-Lagrange equation for arbitrary stationary points of $\mathcal E_\gamma$ under very weak regularity assumptions (we only require the set to have finite perimeter). The Euler-Lagrange equation for stationary points of $\mathcal E_\gamma$ has previously been derived by Choksi and Sternberg \cite{Choksi-Sternberg:2007}, however assuming $C^2$-regularity of the topological boundary. 
An application of our main result will be used in \cite{Goldman:2014} which studies the asymptotics of stationary points of the Ohta-Kawaski energy and its
diffuse interface version. 
\\\\
\noindent
In order to state our main result we need to introduce some notation and specify our hypotheses:

\medskip \noindent
For a given domain $\Omega \subset \mathbb R^n$ with $C^2$-boundary we consider two classes of sets.
\[
\mathcal A:= \{ E \subset \Omega: E\;\text{is bounded and } P(E,\Omega) < +\infty \} \quad\text{and}\quad
\mathcal A_m:= \{ E \in \mathcal A: | E | = m\},
\]
where $m \in (0,\vert \Omega \vert)$.
A stationary point of $\mathcal E_\gamma$ in $\mathcal{A}$ or $\mathcal{A}_m$ is then defined as follows.

\begin{definition}[stationary point of $\mathcal E_\gamma$]\label{defcp}
A set $E \in \mathcal{A}$ is said to be a \emph{stationary point of $\mathcal E_\gamma$ (see \eqref{nonloceqn1}) in $\mathcal{A}$} if for every vector field $X \in C_c^1(\mathbb{R}^n;\mathbb{R}^n)$ with $X\cdot \nu_\Omega =0$ on $\partial \Omega$ we have that
\begin{align}
\label{dervan} \frac{d}{dt}\Big|_{t=0}\mathcal E_{\gamma}(\phi_t(E)) = 0,\end{align}
where $\{\phi_t\}$ is the flow of $X$, i.e. $\partial_t\phi_t =X\circ \phi_t $, $\phi_0 ={\rm id}$.
If \eqref{dervan} holds only for all $X$ such that $\phi_t(E) \in \mathcal{A}_m$ for all $t \in (-\varepsilon,\varepsilon)$ and some small $\varepsilon > 0$, then we call $E$
a \emph{stationary point of $\mathcal E_\gamma$ in $\mathcal{A}_m$}.
\end{definition}
\noindent 
We now specify the assumptions that we impose on the function $G$ appearing in \eqref{nonloceqn1}.

Firstly, we let $\Gamma$ be the fundamental solution of the Laplace operator given by
\[
\Gamma(x,y):=
\begin{cases}
\frac{1}{\omega_{n}(n-2)}\frac{1}{|x-y|^{n-2}}&,n\geq 3\\
-\frac{1}{2\pi}\log|x-y| &,n=2.
\end{cases}
\]
Here $\omega_{n} = \mathcal H^n(\mathbb S^n)$.
We assume that
\[
G(x,y)=
\Gamma(x,y) +R(x,y),
\]
where $R$ is a symmetric corrector function. I.e. 
\[
\begin{cases}
\Delta R(\cdot,y) = \frac{1}{|\Omega|}&\text{in $\Omega$}\\
\frac{\partial R(\cdot,y)}{\partial \nu_\Omega} = - \frac{\partial \Gamma(\cdot,y)}{\partial \nu_\Omega} &\text{on $\partial \Omega$}
\end{cases}
\]
for all $y\in \Omega$. Here we interpret $|\Omega|^{-1}$ to be zero for unbounded domains $\Omega$.
In case $\Omega$ is bounded $G$ is a Neumann Green's function of the Laplace operator.
In case $\Omega = \mathbb R^2$ we also allow that $G(x,y)=\Gamma_\beta(x,y)$ for $\beta \in (0,1)$, where $\Gamma_\beta(x,y):= |x-y|^{-\beta}$. 

For a bounded Borel set $E \subset \Omega$ we define 
\begin{equation}\label{vdef1}
\phi_E(x) := \int_{E} G(x,y) \,dy,\end{equation}
to be the \emph{potential} of $E$ associated to the kernel $G$. 
By standard elliptic theory we have $\phi_E\in C_{loc}^{1,\alpha}( \overline\Omega)$.
\\\\
\noindent
Our main result reads as follows.
\begin{theorem}\label{thm:regularity}
 Let $E$ be a stationary point of the functional $\mathcal E_\gamma$ in $\mathcal{A}$ or $\mathcal{A}_m$ with $f$ and $G$ as above. Then the reduced boundary $\partial_\Omega^*E=\partial^* E\cap \Omega$ is of class  $C^{3,\alpha}$ for all $\alpha \in (0,1)$. In particular, the equation
\begin{equation}\label{EL1} H +2 \gamma \phi_E + f =   \lambda,\end{equation}
holds classically on $\partial_\Omega^*E$ where $H$ is the mean curvature\footnote{by our convention the mean curvature is chosen such that the boundary of the unit ball in $\mathbb R^n$ has positive mean curvature equal to $n-1$} of $\partial_\Omega^*E$,
$\lambda$ is a Lagrange multiplier, and $\phi_E$ is the potential arising from $E$, given by \eqref{vdef1}. (When $E$ is a stationary point in the class $\mathcal{A}$, then $\lambda = 0$.)
The measure $\mu_E =\mathcal H^{n-1}\llcorner \partial_\Omega^*E$ is weakly orthogonal to $\partial \Omega$ in the sense that
\[
\int_{\partial_\Omega^*E}{\rm div}_EX\,d\mathcal H^{n-1} = -\int_{\partial_\Omega^*E}\vec H \cdot X\,d\mathcal H^{n-1} ,
\]
for all $X \in C_c^1(\mathbb R^n;\mathbb R^n)$ with $X\cdot \nu_\Omega = 0$ on $\partial \Omega$.

Moreover, the singular set $(\partial E \setminus \partial^*E)\cap \Omega$ is a relatively closed subset of $\partial E \cap \Omega$ which satisfies $\mathcal H^{n-1}((\partial E \setminus \partial^*E)\cap \Omega)=0$.
\end{theorem}
\begin{remark}
The estimate on the singular set in Theorem~\ref{thm:regularity} is optimal. This can already be seen in the case $\gamma =0$. 
E.g. let $\Omega = B_1(0) \subset \mathbb R^n$ and set $E:=\{x=(x_1,...,x_n)\in \Omega: x_1\cdot x_2>0\}$. Then $E$ is a stationary point of the perimeter functional with singular set $(\partial E \setminus \partial^*E)\cap \Omega = \{ x \in \Omega: x_1= x_2 =0\}$.
\end{remark}

Our paper is organized as follows.
In Section \ref{sec:notation} we introduce our notation and review the basic theory of rectifiable varifolds and sets of finite perimeter, and present Allard's regularity theorem and De Giorgi's structure theorem for the reader's convenience. 
In Section \ref{sec:prelem} we prove some preliminary results that are needed in order to prove Theorem \ref{thm:regularity}.
In Section \ref{sec:proofofmain} we prove Theorem \ref{thm:regularity}.
In Section \ref{sec:bdryregularity} we include, for convenience, the regularity for local minimizers of $\mathcal E_\gamma$ near boundary points $x\in \partial E \cap \partial \Omega$. This has already been proven independently by Julin and Pisante \cite[Theorem 3.2]{Julin-Pisante:2013}.

\section{Notation and preliminaries}\label{sec:notation}
Throughout this work we assume that $\Omega\subset \mathbb R^n$, $n\geq 2$, is a domain (open, connected) of class $C^2$ (although the regularity assumption on the boundary is only needed when we consider vector fields that do not have compact support inside $\Omega$). 
In this section we introduce our notation and summarize basic results from geometric measure theory that are needed in the sequel. For more details on the subject we refer the reader to \cite{Evans-Gariepy:1992,Giusti:1984,Maggi:2012,Simon:1983}.

\subsection{Varifolds and Allard's regularity theorem}
Here we collect basic definitions for varifolds and state Allard's regularity theorem. An $\mathcal H^{k}$-measurable set $M \subset \mathbb R^n$ is called \emph{countably $k$-rectifiable} if
\[ M =  \bigcup_{j=0}^\infty N_j,\]
where $N_j \subset\mathbb R^n$, $0\leq j\leq n-1$, are $k$-dimensional submanifolds of class $C^1$ and $\mathcal{H}^{k}(N_0)=0$. For a vector field $X\in C_c^1(\mathbb R^n;\mathbb R^n)$ we can define the tangential divergence $\div_MX$ of $X$ by setting
\[
\div_MX(x):= \div_{N_j}X(x)\]
for $x\in N_j$, which is well-defined $\mathcal H^k$-a.e. on $M$. Here $\div_{N_j}X(x) = \sum_{i=1}^k\tau_i\cdot DX(x)\tau_i$, where $\{\tau_i\}_{i=1,...,k}$ is an orthonormal basis of the tangent plane $T_xN_j$ of $N_j$ at the point $x$. 

For the purpose of this article we use the following pragmatic definition of rectifiable $k$-varifolds, which usually has to be deduced from the definition (we refer to \cite{Simon:1983} for details):\\
A \emph{rectifiable $k$-varifold} $\mu$ in $\Omega$ is a Radon measure on $\Omega$ such that
\[ \mu = \theta \mathcal H^k \llcorner M,\]
where $M$ is a countably $k$-rectifiable set and where the \emph{multiplicity function} $\theta \in L_{loc}^1(\mathcal H^k \llcorner M)$ is such that $\theta >0$ $\mathcal H^k$-a.e. on $M$.

The \emph{first variation} $\delta \mu$ of $\mu$ with respect to $X\in C^1_c(\Omega,\mathbb R^{n})$ is given by
\[
\delta \mu(X) := \int_M \div_MX\,d\mu,\]
which by \cite[\S 16]{Simon:1983} is equal to $ \frac{d}{dt}(\phi_t{}_\sharp\mu)(\Omega) |_{t=0} $. Here $\phi_t{}_\sharp\mu$ denotes the \emph{image varifold} given by $\phi_t{}_\sharp\mu: = (\theta\circ\phi_t^{-1}) \mathcal H^k \llcorner \phi_t(M)$, and where $\{\phi_t\}$ denotes the flow of $X$.

We say that $\mu$ has \emph{generalized mean curvature} $\vec H$ in $\Omega$ if 
\begin{equation}\label{genH}
\delta \mu(X)=\int_M\div_MX\,d\mu = - \int_M \vec H \cdot X\,d\mu\quad\text{for all $X\in C^1_c(\Omega;\mathbb R^n)$},
\end{equation}
where $\vec H$ is a locally $\mu$-integrable function on $M\cap \Omega$ with values in $\mathbb R^{n}$. 
We remark that using the Riesz representation theorem such an $\vec H$ exists if the total variation $\|\delta \mu \|$ is a Radon measure in $\Omega$ and moreover $\|\delta V\|$ is absolutely continuous with respect to $\mu$ (see \cite{Simon:1983} for details).

We make the trivial but important remark that a rectifiable $k$-varifold $\mu$ in $\Omega$ that has finite total mass $\mu(\Omega)$ naturally defines a rectifiable $k$-varifold in $\mathbb R^n$.

A fundamental result in the theory of varifolds is the following regularity theorem due to Allard \cite{Allard:1972} (see also \cite[Chapter 5]{Simon:1983} for a more accessible approach) that holds for rectifiable $k$-varifolds $\mu$ in $\Omega \subset \mathbb R^n$. We use the following hypotheses.
\begin{equation}\label{hyp}
\left.\begin{split} 1\leq \theta\,\,\, \mu\text{-a.e. ,  }0\in\spt(\mu)\,\,,&\,B_\rho(0)\subset \Omega\\
\alpha_k^{-1}\rho^{-k}\mu(B_\rho(0))\leq & 1+\delta\\
\left(\int_{B_\rho(0)}|\vec H|^p\, d\mu\right)^\frac{1}{p}\rho^{1-\frac{k}{p}} &\leq \delta .
\end{split}\,\,\,\,\right\}\tag{h}
\end{equation}

\begin{theorem}[Allard's Regularity Theorem]\label{thm:Allard} 
For $p>k$, there exist $\delta=\delta(n,k,p)$ and $\gamma=\gamma(n,k,p)$ $\in (0,1)$ such that if $\mu$ is a rectifiable $k$-varifold in $\Omega$ that has generalized mean curvature $\vec H$ in $\Omega$ (see \eqref{genH}) and satisfies hypotheses \eqref{hyp},
then ${\rm spt}(\mu)\cap B_{\gamma\rho}(0)$ is a graph of a $C^{1, 1-\frac{k}{p}}$ function with scaling invariant $C^{1, 1-\frac{k}{p}}$~estimates depending only on $n,k,p,\delta$.
\end{theorem}
\begin{remark}
More precisely, 
there is a linear isometry $q$ of $\mathbb{R}^{n}$ and a function $u \in C^{1,1-\frac{k}{p} }(B_{\gamma r}^k(0);\mathbb R^{n-k})$
with $u(0) = 0$, $\textrm{spt}(\mu) \cap B_{\gamma \rho}(0) = q(\textrm{graph}(u)) \cap B_{\gamma \rho} (0)$, and
\begin{equation}\label{estimates}
 \rho^{-1} \sup_{B_{\gamma \rho}^k(0)}|u| + \sup_{B_{\gamma \rho}^k(0)} |Du| + \rho^{1-\frac{k}{p} } \sup_{\substack{x,y \in B_{\gamma \rho}^k (0)\\ x \neq y}} |x-y|^{-(1- \frac{k}{p})} |Du(x)-Du(y)| \leq c(n,k,p)\delta^{1/4k}.
\end{equation}
\end{remark}

\subsection{Sets of finite perimeter}
Let $E \subset \Omega$ be a Borel set. We say that
$E$ has finite perimeter $P(E,\Omega)$ in $\Omega$ if 
\[
P(E,\Omega) :=\sup_{\substack{X \in C_c^1(\Omega;\mathbb{R}^n)\\ |X | \leq 1}} \int_{E} {\rm div} X\,dx  < \infty.
\]
The Riesz representation theorem implies the existence of a Radon measure $\mu_E$ on $\Omega$ and a $\mu_E$-measurable vector field $\eta_E:\Omega \to \mathbb R^n$ with $|\eta_E|=1$ $\mu_E$-a.e. such that
\[ \int_E \div X\,dx = \int_{\mathbb{R}^n} X\cdot \eta_E \,d\mu_E \quad\textrm{for all } X \in C_c^1(\Omega;\mathbb{R}^n).\]
The vector valued measure $\vec \mu_E := \eta_E\,\mu_E$ is sometimes referred to as the Gauss-Green measure of $E$ (with respect to $\Omega$). For the total perimeter of the set $E$ in $\Omega$ we have
\[ P(E,\Omega) = \mu_E(\Omega).\]

\noindent In the case that $\partial E\cap \Omega$ is of class $C^1$, we have
\[
 \vec \mu_E = \nu_E \mathcal{H}^{n-1} \llcorner (\partial E\cap\Omega)\quad\text{and}\quad P(E,\Omega) = \mathcal{H}^{n-1}(\partial E\cap \Omega).
\]
In particular, we have for every point $x\in \partial E\cap \Omega$
\begin{equation}\label{LimitNormal}
\nu_E(x) = \lim_{r \to 0} \dashint_{ \partial E \cap B_r(x)} \nu_E \,d\mathcal{H}^{n-1} = \lim_{r \to 0} \frac{\vec \mu_E(B_r(x))}{\mu_E(B_r(x))}.
\end{equation}
For a generic set $E$ of finite perimeter, the \emph{reduced boundary} $\partial_\Omega^*E $ of $E$ in $\Omega$ is defined as those $x \in \partial E \cap \Omega$ such that the above limit on the right hand side exists and has norm $1$. The Lebesgue-Besicovitch differentiation theorem implies that $\mu_E(\mathbb R^n\setminus \partial_\Omega^*E)=0$.
The vector field $\nu_E \in L^1(\mu_E; \mathbb R^{n})$ \emph{defined} by the equation \eqref{LimitNormal} on $\partial_\Omega^*E$ (and set to $0$ elsewhere), is called the \emph{measure theoretic outer unit normal} of $E$. 
For more details on sets of finite perimeter we refer to \cite{Giusti:1984,Simon:1983,Evans-Gariepy:1992}.

\begin{theorem}[De Giorgi's structure theorem] \label{thm:DeGiorgi}
 Suppose $E$ has finite perimeter in $\Omega$. Then $\partial_\Omega^*E$ is countably $(n-1)$-rectifiable. % and% $\mu_E = \mathcal{H}^1 \lefthalfcup \partial_\Omega^*E$.
 In addition for all $x \in \partial_\Omega^*E$
\begin{equation}
\Theta(\mu_E,x) := \lim_{r \to 0} \frac{\mu_E(B_r(x))}{\alpha_{n-1} r^{n-1}} = 1,
\end{equation}
where $\alpha_{n-1}$ is the volume of the unit ball in $\mathbb{R}^{n-1}$. (i.e. the limit exists and is equal to $1$.) 
Moreover, $\mu_E = \mathcal{H}^{n-1}\llcorner \partial_\Omega^*E$. 
 \end{theorem}
\begin{remark}\label{rmk:DeGiorgi}
De Giorgi's structure theorem in particular shows that every set $E$ of finite perimeter defines - through its generalized surface measure $\mu_E$ - a rectifiable $(n-1)$-varifold of multiplicity $\theta \equiv 1$ on $\partial_\Omega^*E$.
\end{remark}
Let $E\subset \Omega$ be of finite perimeter in $\Omega$. If $\Omega$ is Lipschitz regular, one can define the (inner) \emph{trace} $\chi_E^+ \in L^1(\mathcal H^{n-1}\llcorner \partial \Omega)$ of $\chi_E$ on $\partial \Omega$. For details we refer to \cite[Chapter 5.3]{Evans-Gariepy:1992}. For every vector field $X \in C_c^1(\mathbb R^n;\mathbb R^n)$ we have
\[
\int_E \div X\,dx = \int_{\partial_\Omega^*E}X\cdot\nu_E\,d\mathcal H^{n-1} +  \int_{\partial\Omega} X\cdot\nu_\Omega\,\chi_E^+\,d\mathcal H^{n-1}.
\]
This implies that $E$ is also a set of finite perimeter as a subset of $\mathbb R^n$ with
\[ P(E,\mathbb R^n) = P(E,\Omega)+ \int_{\partial\Omega} |\chi_E^+|\,d\mathcal H^{n-1}.\]
As a finite perimeter set in $\mathbb R^n$, $E$ also has a Gauss-Green measure which we shall denote by $ \vec \mu_{E}^*$. Obviously $ \vec \mu_{E}^* \llcorner \Omega =  \vec \mu_E$ and $\partial^*E \cap \Omega = \partial_\Omega^*E$.
\\\\

Since sets of finite perimeter are equivalence classes of sets, one needs to choose a good representative in order to talk about their regularity properties. W.l.o.g. (see \cite[Proposition 3.1]{Giusti:1984} for details) we will always assume that any finite perimeter set $E$ at hand satisfies the following properties:
\begin{align}\label{representative} 
(a)\quad& E \text{  is Borel} \nonumber \\
(b)\quad& 0 < \vert  E \cap B_\rho(x) \vert < \vert  B_\rho(x) \vert\;\;\text{for all}\;\;x \in \partial  E\;\;\text{and all}\;\;\rho > 0 \\
%$E_0 \cap E = \emptyset$ and $E_1 \subset E $, where
%\[
%E_\theta := \{x \in \mathbb R^n:\exists \rho > 0\;\text{s.t.}\;\vert E\cap B_\rho(x)\vert=\theta \vert B_\rho(x)\vert\}\quad\theta\in \{0,1\}.
%\]
(c)\quad& \overline{\partial_\Omega^*E} = \partial E\cap \overline \Omega \text{ which implies that ${\rm spt}(\mu_E)=\partial E \cap \overline \Omega.$} \nonumber 
\end{align}

\subsection{The first variation of perimeter}
Let $X \in C_c^{1}(\Omega; \mathbb{R}^n)$ with corresponding flow $\{\phi_t\}$.
The \emph{first variation of perimeter} is then easily computed as (see  \cite{Giusti:1984,Maggi:2012})
\begin{equation}\label{pervar}
\frac{d}{dt}\Big|_{t=0}P(\phi_t(E),\Omega) = \int_{\partial_\Omega^*E} \textrm{div}_{E} X \,d\mathcal{H}^{n-1},
\end{equation}
where $\textrm{div}_E X$ is the tangential divergence of the vector field $X$ with respect to $E$:
\[ \textrm{div}_E X = \textrm{div} X - \nu_E\cdot DX\nu_E,\]
which obviously agrees with the definition of tangential divergence with respect to $\partial_\Omega^*E$. Hence, the expression \eqref{pervar} equals the first variation $\delta \mu_E(X)$ of the varifold $\mu_E$ with respect to $X$.

In order to investigate the behavior of stationary points $E$ of $\mathcal E_\gamma$ at the boundary $\partial \Omega$ of $\Omega$ we need to allow for more general variations (as already appearing in Definition~\ref{defcp}). By the regularity assumption on $\partial \Omega$ it follows (cf. \cite{Grueter:1987}) that $\phi_t(\Omega)\equiv \Omega$ (and $\phi_t(\partial \Omega)\equiv \partial \Omega$) for the flow $\{\phi_t\}$ of any vector field $X\in C_c^1(\mathbb R^n;\mathbb R^n)$ such that $X\cdot \nu_\Omega =0$ on $\partial \Omega$.
Since $P(\phi_t(E),\Omega)\equiv (\phi_t{}_\sharp\mu_E)(\mathbb R^n)$
we see that the formula \eqref{pervar} still holds for such vector fields $X$.

\section{Preliminary results}\label{sec:prelem}

\begin{proposition}[First variation of nonlocal perimeter]\label{1st variation of perimeter}
Let $E \in \mathcal A$ and let $X\in C_c^1(\mathbb R^n;\mathbb R^n)$ with $X\cdot \nu_\Omega =0$ on $\partial \Omega$ be a vector field with corresponding flow $\{\phi_t\}$.
Then
\begin{align*}
\frac{d}{dt}\mathcal E_\gamma(\phi_t(E)) \Big|_{t=0} =\int_{ \partial_\Omega^*E}\div_E X\,d\mathcal{H}^{n-1} +2\gamma\int_{\partial_\Omega^*E}  \phi_E X \cdot \nu_E \,d\mathcal{H}^{n-1} + \int_{\partial_\Omega^*E} f X \cdot \nu_E \,d\mathcal{H}^{n-1}.
\end{align*}
\end{proposition}
\begin{proof}
The first variation of perimeter is equation \eqref{pervar}. It remains to compute the first variation of the nonlocal term; the computation of the first variation
of the third term is similar but easier. By the change of variables formula it holds that
\begin{align}\label{G1}
 \int_{\phi_t(E)}\int_{\phi_t(E)}G(x,y)\,dx\,dy = \int_{E}\int_{E}G(\phi_t(x),\phi_t(y))\,\vert\det D\phi_t(x)\vert\vert\det D\phi_t(y)\vert\,dx\,dy .
\end{align}
Hence, we compute using \eqref{G1} and the assumptions on $G$ which allow us to differentiate under the integral
\begin{align}\label{1stvarcomp}
\frac{d}{dt}\Big|_{t=0} &\int_{\phi_t(E)}\int_{\phi_t(E)}G(x,y)\,dx\,dy\nonumber\\
& =2\int_{E}\int_{E} (\nabla_xG)(x,y) \cdot X(x)dx\,dy  +  2\int_{E}\int_{E} G(x,y)\,\text{div}X(x) \,dx\,dy \nonumber \\
& =2\int_{E}\int_{E} \div ( G(\cdot,y)X)(x) \,dx\,dy \nonumber\\
& =2\int_{E}\int_{E} \div ( \Gamma(\cdot,y)X)(x) \,dx\,dy +2\int_{E}\int_{\partial_\Omega^*E} R(\cdot,y)\,X\cdot\nu_E \,d\mu_E\,dy .
\end{align}
We cannot directly apply the divergence theorem to the first term of \eqref{1stvarcomp} since $\partial_\Omega^*E$ is only
$(n-1)$-rectifiable and $G(\cdot,y)$ is not of class $C^1$ near $x=y$. We can get around this technical obstacle by applying the results of \cite{Chen-Torres-Ziemer:2009}, but we present a simple argument which suffices in our case. Since $\Omega$ is of class $C^2$ (for this argument Lipschitz is enough) we have (see Section~\ref{sec:prelem}) that $E$ is a set of finite perimeter in $\mathbb R^n$. 
By \cite[Theorem 1.24]{Giusti:1984} we can approximate $E$ in the support
of $X$ by smooth sets $E^{i}$ such that
\begin{align*}
\chi_{E^{i}} &\to \chi_E \textrm{ in } L^1_{loc}(\mathbb R^n)\textrm{ and } \\
\vec\mu_{E^{i}}^* &\to \vec\mu_E^* \textrm{ weakly as Radon measures on $\mathbb R^n$}.
\end{align*}
We may apply
the Lebesgue dominated convergence theorem to conclude that
\begin{align}\label{dominated}
\int_{E}\int_{E} \div ( \Gamma(\cdot,y)X)(x) \,dx\,dy= \lim_{i \to \infty} \int_{E}\int_{E^i} \div ( \Gamma(\cdot,y)X)(x) \,dx\,dy.
\end{align}
Moreover, we have
\begin{align}\label{dominated1}
\int_{E^{i}} \div ( \Gamma(\cdot,y)X)(x) \,dx =\lim_{\rho \to 0} \int_{E^{i} \setminus  B_{\rho}(y)} \div ( \Gamma(\cdot,y)X)(x) \,dx.
\end{align}
We may now apply the divergence theorem and we have for a.e. $0< \rho < 1$
\begin{align}\label{dominated2}
\int_{E^{i} \backslash B_{\rho}(y)} \div ( \Gamma(\cdot,y)X)(x) \,dx 
&= \int_{\partial E^i \setminus B_{\rho}(y)} \Gamma(\cdot,y) \,X \cdot \nu_{E^{i}}\,d\mathcal H^{n-1} \nonumber\\
&\quad- \int_{\partial B_{\rho}(y) \cap E^{i}} \Gamma(\cdot,y) \,X \cdot \nu_{B_{\rho}(y)} \,d\mathcal H^{n-1}.
\end{align}
The second term on the right hand side of \eqref{dominated2} can be estimated by
$c(n) \sup |X|\,\rho^{1-\varepsilon}$ for some $\varepsilon \in [0,1)$, and hence goes to zero as $\rho \to 0$.
On the other hand, we have
\begin{align*}
&\left| \int_{\partial E^i \setminus B_\varrho(y)} \Gamma(\cdot,y)X \cdot\nu_{E^i}\,d\mathcal H^{n-1} - \int_{\partial E^i } \Gamma(\cdot,y)X \cdot\nu_{E^i}\,d\mathcal H^{n-1} \right| \\
& \quad \leq \mathcal H^{n-1}(\partial E^i  \cap B_\varrho(y) )^{1-\frac{1}{p} } \left( \int_{\partial E^i \cap {\rm spt}(X) } |\Gamma(\cdot,y)|^p\,d\mathcal H^{n-1} \right)^\frac{1}{p} \sup |X|  ,
\end{align*}
where $p\in (1, \frac{n-1}{n-2})$, in case $n\geq 3$, and $p \in (1,\beta^{-1})$ in case $G=\Gamma_\beta$. Whence, upon combining 
\eqref{dominated1} and \eqref{dominated2},
\[
\int_{E^{i}} \div ( \Gamma(\cdot,y)X)(x) \,dx=\int_{\partial E^i } \Gamma(\cdot,y) X \cdot \nu_{E_{i}} \,d\mathcal H^{n-1}.
\]
Using \eqref{1stvarcomp} and applying Fubini's theorem we arrive at
\begin{align}\label{dominated4}
\int_E \int_{E^{i}} \div ( G(\cdot,y)X)(x) \,dx \,dy = \int_{\partial E^i}  \phi_E X \cdot \nu_{E^{i}} \,d\mathcal H^{n-1}.
\end{align}
Now let $i \to \infty$, using the fact that $\phi_E$ is continuous and that $X$ has compact support, and combining \eqref{dominated} and \eqref{dominated4} we obtain
\begin{align*}
\int_{E}\int_{E} \div ( G(\cdot,y)X)(x) \,dx\,dy=  \int_{\partial^*E}  \phi_E X \cdot \nu_E \,d\mathcal{H}^{n-1}.
\end{align*}
The claim now follows from the fact that $X$ is tangential to $\partial \Omega$.
\end{proof}

\begin{lemma}\label{lagrange1}
There exists a vector field $Y \in C_c^1(\Omega;\mathbb R^n)$ such that $\int_{E}\div Y dx =1$.
\end{lemma}
\begin{proof}
Assume by contradiction that for every vector field $X \in C_c^1(\Omega;\mathbb R^n)$: $\int_{E}\text{div}X dx =0$. Then by Du Bois-Reymond's lemma \cite{Giaquinta-Hildebrandt:1996-I} we conclude that 
$$\chi_E=0 \;\;\;\text{or}\;\;\; \chi_E=1 \;\;\;\text{$\mathcal L^n$-a.e. on $ \Omega$},$$
where we used that $\Omega$ is connected. Hence,
$$E = \Omega \;\;\;\text{or}\;\;\; E =\emptyset \;\;\;\text{in the measure theoretic sense}.$$
This contradicts the assumption that $0<\vert E  \vert < \vert \Omega \vert$, proving the claim.
\end{proof} 

\begin{proposition}[Euler-Lagrange equation of non-local perimeter] \label{firstvariation} Let $E$ be a stationary point of $\mathcal E_\gamma$ in $\mathcal{A}_m$ or $\mathcal{A}$. Then there exists a real number $\lambda$ such that $\mu_E$
has a generalized mean curvature vector
\begin{equation}
 \vec H = -(\lambda - 2 \gamma \phi_E - f)\nu_E
\end{equation}
and such that $\mu_E$ is weakly orthogonal to $\partial \Omega$.
That is, for every vector field $X \in C_c^1(\mathbb{R}^n;\mathbb{R}^n)$ with $X\cdot \nu_\Omega =0$ on $\partial\Omega$ the following variational equation is true:
\begin{equation}\label{guy} \int_{\partial_\Omega^*E} \textrm{\em div}_E X \,d\mathcal{H}^{n-1} = - \int_{\partial_\Omega^*E} \vec H \cdot X \,d\mathcal{H}^{n-1}.\end{equation}
For stationary points in $\mathcal{A}$, we have $\lambda = 0$. 
\end{proposition}
\begin{proof}
\noindent {\emph{Step 1: Construction of the local variation}.}\\
The case of variations in $\mathcal{A}$ is an immediate consequence of Proposition \ref{1st variation of perimeter}. 
For the case of $\mathcal{A}_m$ let $Y \in C_c^1(\Omega;\mathbb R^n)$ be a vector field such that $\int_{E}\div Y dx =1$. The existence of such a vector field is guaranteed by Lemma~\ref{lagrange1}. 
Let $\{\phi_t\}$ be the flow of $X$ and $\{\psi_s\}$ the flow of $Y$. For $(t,s) \in \mathbb R^2$ set
$$ A(t,s):=P(\psi_s(\phi_t(E))) \cap \Omega)$$
and
$$\mathcal V(t,s):=\vert\psi_s(\phi_t(E))\vert -\vert E \vert .$$
Then $\mathcal V \in C^1(\mathbb R^2)$, $\mathcal V(0,0)=0$ and $\partial_s\mathcal V(0,0)=\int_{E}\div Y dx =1$. The implicit function theorem ensures the existence of an open interval $I$ containing $0$ and a function $\sigma\in C^1(I)$ such that 
$$\mathcal V(t,\sigma(t))=0\;\;\;\text{for all $t \in I$ and}\;\;\; \sigma'(0)=-\frac{\partial_t\mathcal V(0,0)}{\partial_s\mathcal V(0,0)}.$$
Hence, 
$$t \mapsto \psi_{\sigma(t)}\circ\phi_t$$
is a 1-parameter family of $C^1$-diffeomorphisms of $\overline\Omega$ and thus defines a volume preserving variation of $E$ in $\Omega$.\\

\noindent {\emph{Step 2: Computing the first variation}.}\\
The fact that $E$ is a stationary point in the class $\mathcal{A}_m$ then implies from Proposition \ref{1st variation of perimeter}
$$\frac{d}{dt}\Big|_{t=0} A(t,\sigma(t)) = \int_{\partial_\Omega^*E} \textrm{div}_E X= - 2\gamma\int_{\partial_\Omega^*E}  \phi_E X \cdot \nu_E \,d\mathcal{H}^{n-1} - \int_{\partial_\Omega^*E} f X \cdot \nu_E  \,d\mathcal{H}^{n-1}.$$
On the other hand, we have
\begin{align*}
\frac{d}{dt}\Big|_{t=0} A(t,\sigma(t)) & = \partial_t\mathcal A(0,0)+\sigma'(0)\partial_s\mathcal A(0,0) \\
& =\int_{\partial_\Omega^*E}\textrm{div}_E X \,d\mathcal{H}^{n-1} + \sigma'(0)\int_{\partial_\Omega^*E}\textrm{div}_E Y \,d\mathcal{H}^{n-1} \\
& =\int_{\partial_\Omega^*E}\textrm{div}_E  X \,d\mathcal{H}^{n-1} - \frac{\int_{E}\div X dx}{\int_{E}\div Y dx}\int_{\partial_\Omega^*E}\textrm{div}_E Y \,d\mathcal{H}^{n-1} \\
& =\int_{\partial_\Omega^*E}\textrm{div}_E  X \,d\mathcal{H}^{n-1} - \lambda \int_{\partial_\Omega^*E} X \cdot \nu_E \,d\mathcal{H}^{n-1},
\end{align*}
where $\lambda:=\int_{\partial_\Omega^*E}\textrm{div}_E Y \,d\mathcal{H}^{n-1}$, and where we used the divergence theorem on the last line. Therefore, setting $\vec H: = (2\gamma \phi_E +f -\lambda )\nu_E$, we have
$$\int_{\partial_\Omega^*E}\div_E X \,d\mathcal{H}^{n-1}= -\int_{\partial_\Omega^*E}  \vec H \cdot X \,d\mathcal{H}^{n-1}$$
for every vector field $X \in C_c^1(\mathbb R^n;\mathbb R^n)$ with $X \cdot \nu_\Omega =0$ on $\partial \Omega$. 
\end{proof}

\section{Proof of the Theorem \ref{thm:regularity}}\label{sec:proofofmain}

Firstly, notice that the weak orthogonality of $\mu_E$ and $\Omega$ is included in Proposition~\ref{firstvariation}.
\\\\
We want to apply Allard's regularity theorem (here Theorem~\ref{thm:Allard}) to establish the regularity of the reduced boundary $\partial_\Omega^*E$ . We verify the necessary hypotheses:

By De Giorgi's structure theorem (here Theorem \ref{thm:DeGiorgi}) and Remark~\ref{rmk:DeGiorgi}  we have that $\mu_E$ is a multiplicity-$1$ rectifiable $(n-1)$-varifold. Moreover, for each point $x \in \partial_\Omega^*E$ we have that $\Theta(\mu_E,x) = 1$. Now, we choose any point $x_0 \in \partial_\Omega^*E$. W.l.o.g., after possibly translating and rotating the set $E$, we may assume that $x_0 = 0$ and $\nu_E(0)= - e_n$. We fix any $p>n-1$ and pick $\delta \in (0,1)$ to be as in the statement of Theorem~\ref{thm:Allard}. Since $\Theta(\mu_E,0) = 1$ we can find a small radius $\rho>0$ such that
\begin{equation}\label{verifyingassumptions}
B_\rho(0) \subset \subset \Omega\quad\text{and}\quad \alpha_{n-1}^{-1} \rho^{-{n-1}} \mu_E(B_\rho(0)) \leq 1 + \delta.
\end{equation}
Proposition~\ref{firstvariation} implies that $\mu_E$ has generalized mean curvature $\vec H$ in $\Omega$, given by
\[
\vec H= -(\lambda - 2 \gamma \phi_E - f)\nu_E
\]
for some constant $\lambda \in \mathbb R$.
We have
\[
\|\vec H \|_{L^\infty(\mu_E\llcorner B_\rho(0))} \leq |\lambda|+2\gamma \sup_{B_\rho(0)}|\phi_E|+\sup_{B_\rho(0)}|f| =: c_0.
\]
With H\"older's inequality and \eqref{verifyingassumptions} we get
\begin{align*}
\left(\int_{B_\rho(0)} |\vec H|^p\,d\mu_E \right)^\frac{1}{p} \rho^{1-\frac{n-1}{p} }
& \leq c_0\, (1 + \delta)^\frac{1}{p} \alpha_{n-1}^\frac{1}{p} \,\rho,
\end{align*}
which is less that $\delta$ provided $\rho \leq \delta c_0^{-1}\, 2^{-\frac{1}{p}} \alpha_{n-1}^{-\frac{1}{p}}$.
Thus the hypotheses \eqref{hyp} are satisfied and Theorem \ref{thm:Allard} implies the existence of a function $u:B'( := B_{\gamma\rho}^{n-1}(0)) \to \mathbb R$ of class $C^{1,\alpha}$, $\alpha= 1- (n-1)/p$, such that $u(0) = 0$, $Du(0)=0$, and ${\rm spt}(\mu_E) \cap B_{\gamma \rho}(0) = {\rm graph}(u) \cap B_{\gamma \rho} (0)$.
Moreover, our orientation assumption on $E$ implies that $\overline E \cap (B' \times I) = \overline{\textrm{epigraph}(u) }\cap (B' \times I)$ for some open interval $0\in I$.

Now let $X(x',z) = \zeta(z) \eta(x') e_n$, where $\eta \in C_c^1(B')$, $x' \in B'$, $e_n=(0,...,0,1)$ is the $n$-th-standard basis vector, and where $\zeta \in C_c^\infty(\mathbb R)$ is a cut-off function such that $(\zeta\circ u)(x)=1$ for every $x \in B'$.

Then recalling that $\textrm{div}_{E} X = \div X - \nu_E \cdot D X \nu_E$, we have $\textrm{div}_E X = - (\nabla' \eta,0) \cdot \nu_E  \nu_E^n$ where $\nu_E^n$ is the $n$-th component of the normal vector, and where $\nabla'$ is the gradient in $\mathbb{R}^{n-1}$. Since $\partial^* E \cap (B' \times I)= \partial E \cap (B' \times I)$ is the graph of $u$, and by our orientation assumption, we 
have that $\nu_E^n = \frac{-1}{\sqrt{1 + |\nabla' u|^2}}$. Using the area formula, equation \eqref{guy} becomes
\begin{equation}\label{weakform}
 -\int_{B'} \frac{\nabla'\eta \cdot \nabla' u }{ \sqrt{1+|\nabla' u|^2}} \,dx' =\int_{B'} ( \lambda -2\gamma \,v_E(x',u) - f(x',u) ) \,\eta\,dx'.
\end{equation}
Equation \eqref{weakform} is the weak form of the prescribed mean curvature equation.
Since by Theorem~\ref{thm:Allard} the gradient of $u$ is locally uniformly bounded in $C^{0,\alpha}$ and since the right hand side of \eqref{weakform} is of class $C^{1,\alpha}$, interior Schauder estimates (see \cite{Gilbarg-Trudinger:2001}) and bootstrapping imply local $C^{3,\alpha}$ regularity of the function $u$. Thus \eqref{weakform} holds pointwise, and since $x_0 \in \partial_\Omega^*E$ was arbitrary we have
\[ H+ 2\gamma \phi_E +f=\lambda  \textrm{ on } \partial_\Omega^*E, \]
where $H$ is the classical mean curvature of the surface $\partial_\Omega^*E$.

\subsection{On the size of the singular set}

By a direct consequence the monotonicity formula, see \cite[Corollary 17.8]{Simon:1983}, we have that $\Theta(\mu_E,x)$ exists and that $\Theta(\mu_E,x) \geq 1$ for \emph{every} point $x \in {\rm spt}(\mu_E)=\overline{\partial E\cap \Omega}$.
This allows us to estimate the size of the singular set $(\partial E\setminus\partial^*E)\cap\Omega$.

\begin{proposition}\label{estimateondimofsingset1}
We have the following estimate
$$\mathcal H^{n-1}((\partial E\setminus\partial^* E)\cap \Omega)=0.$$
\end{proposition}
\begin{proof}W.l.o.g. we may assume that $\Omega$ is bounded. Otherwise, we may exhaust $\Omega$ with bounded sets. We know that $\mu_{E}=\mu_{E} \llcorner \partial_\Omega^*E$. Hence,
$$\mu_{E}((\partial E\setminus\partial^* E)\cap \Omega)=0.$$
Since $\mu_{E}$ is a Radon measure, given an $\varepsilon>0$, there exists an open set $U_{\varepsilon} \subset \Omega$  containing $(\partial E\setminus\partial^* E)\cap \Omega$ such that
$$\mu_{E}(U_{\varepsilon})\leq\varepsilon.$$
Now, by our w.l.o.g.-assumptions: $\overline{\partial_\Omega^*E}=\partial E\cap \overline \Omega$, and so for any fixed $\delta>0$ 
$$(\partial E\setminus\partial^* E)\cap \Omega\subset \bigcup\mathcal F,$$
where $\mathcal F:=\{ B_{\rho}(x)\subset\subset U_{\varepsilon}: x\in\partial_\Omega^*E\;\;\text{and}\;\;\rho\leq \delta\}$. By Vitali's covering theorem there exists a countable family $\mathcal G \equiv \{B_{\rho_{j}}(x_j):j\in \mathbb N\}$ of disjoint balls in $\mathcal F$ such that
$$\bigcup \mathcal F \subset \bigcup_{j=1}^\infty\overline B_{5\rho_j}(x_j).$$
Hence,
$$\sum_{j=1}^\infty\mu_{E}( B_{\rho_j}(x_j))\leq \mu_{E}(U_{\varepsilon})\leq  \varepsilon.$$
On the other hand, by the monotonicity formula \cite[Theorem 17.7]{Simon:1983},
\begin{align*}
	\mu_{E}( B_{\rho_j}(x_j))\geq \frac{\alpha_{n-1}}{2}\rho_j^{n-1},
\end{align*}
if $\delta$ is small enough as to guarantee that 
$$\alpha_{n-1}^{\frac{n-1}{p}}\left(1- 2^{-\frac{1}{p}} \right) \geq \frac{\| \vec H\|_{L^p}}{p-(n-1)}\delta^{1-\frac{n-1}{p}} ,$$
for some $p>n-1$.
Therefore,
$$ \alpha_{n-1}\sum_{j=1}^\infty \rho_j^{n-1}\leq 2\varepsilon, $$
which yields
\begin{align*}
\mathcal H_{5\delta}^{n-1}((\partial E\setminus\partial^* E)\cap \Omega) & \leq \alpha_{n-1}\sum_{j=1}^\infty(5\rho_j)^{n-1} \leq 5^{n}	\varepsilon,
\end{align*}
for any $\varepsilon>0$ and $\delta>0$. Letting $\varepsilon,\delta\searrow 0$ we conclude
$$\mathcal H^{n-1}((\partial E\setminus\partial^* E)\cap \Omega) =0.$$
\end{proof}
\begin{remark}
It is an interesting question whether the estimate on the Hausdorff dimension of the singular set can be improved under the additional assumption of stability.
Even without the nonlocal term this is an open problem in the class $\mathcal A_m$. For the case of minimal surfaces Wickramasekera \cite{Wickramasekera:2014} recently showed that in this case the singular set has Hausdorff dimension at most $n-8$.
\end{remark}
\section{Boundary regularity of local minimizers}\label{sec:bdryregularity}
In this section we outline how Theorem~\ref{thm:regularity} can be used to prove boundary regularity, that is regularity near points $x\in \partial \Omega\cap \partial E$, for local minimizers $E$ of $\mathcal E_\gamma$ in $\mathcal A$ or $\mathcal A_m$. This has already been established in \cite{Julin-Pisante:2013} but we include it for convenience of the reader. 

As mentioned earlier, the interior regularity for local minimizers of $\mathcal E_\gamma$ was proved by Sternberg and Topaloglu \cite[Propostion 2.1]{Sternberg-Topaloglu:2011}. The authors prove that local minimizers of $\mathcal E_\gamma$ are $(K,\varepsilon)$-minimal and can thus appeal to the standard methods (cf. \cite{Massari:1974}).
We include a slightly different proof. 

\begin{definition}\label{def:locmin}
We say that $E \in\mathcal A$ or $\mathcal A_m$ is a \emph{local minimizer} of $\mathcal E_\gamma$ in $\mathcal A$ or $\mathcal A_m$ (at scale $R$) if 
for all balls $B_R(x)\subset\mathbb R^n$ we have that
\begin{equation}\label{locmin}
\mathcal E_\gamma(E) \leq \mathcal E_\gamma(F)\quad\text{for all $F\in \mathcal A$ or $\mathcal A$ with $E\Delta F\subset\subset B_R(x)$}.
\end{equation}
\end{definition}

\begin{remark}\label{intextpoints}
Theorem~\ref{thm:regularity} implies that for any ball $B_\rho(x)\subset \mathbb R^n$ with $0<|E\cap B_\rho(x)|<|\Omega \cap B_\rho(x)|$ we can find exterior and interior points, i.e. there exist two balls $B_r(a),B_r(b)\subset\subset \Omega \cap B_\rho(x)$ with $\bigcup_{t\in [0,1]} B_r(ta+(1-t)b)\subset\subset \Omega$ such that
\[
|B_r(a)\setminus E|=|E \cap B_r(a)|=0.
\]
\end{remark}

We are now ready to prove the following 

\begin{proposition}[cf. {\cite[Propostion 2.1]{Sternberg-Topaloglu:2011}} ]\label{prop:quasimin}
Let $E \in\mathcal A$ or $\mathcal A_m$ be a local minimizer of $\mathcal E_\gamma$ in $\mathcal A$ or $\mathcal A_m$ at scale $2R_0 >0$, and let $0<|E\cap B_{R_0}(x_0)|<|\Omega \cap B_{R_0}(x_0)|$ for some ball $B_{R_0}(x_0) \subset \mathbb R^n$. Then $E$ is $(K,\varepsilon)$-minimal in $B_R(x_0)$ for some $R\leq R_0$, that is for every $B_\rho \subset\subset B_R(x_0)$
\[
P(E,\Omega )\leq P(F,\Omega) +K\rho^{n} \quad\text{for all $F$ such that $F\Delta E \subset\subset B_\rho$.}
\]
\end{proposition}
\begin{proof}
Let $B_\rho \subset\subset B_R(x_0)$ and let $F$ be such that $F\Delta E \subset\subset B_\rho$. 
We only give a proof for local minimizers in $\mathcal A$. (For the case with a volume constraint one may use Remark~\ref{intextpoints} to adjust the volume of the competitor $F$ which gives us the additional term $\frac{c(n)}{r}\rho^n$ on the right hand side of equation \eqref{minprop} below. We refer to \cite{Gonzalez-Massari-Tamanini:1983} for details. Alternatively, one can proceed as in \cite{Sternberg-Topaloglu:2011} and use a result of Giusti \cite[Lemma 2.1]{Giusti:1981} to balance out the volume constraint.)

By \eqref{locmin} we have that
\begin{align}\label{minprop}
P(E,\Omega )
&\leq P(F,\Omega)  +\gamma \int_F\int_FG(x,y)\,dx\,dy -\gamma \int_E\int_E G(x,y)\,dx\,dy \\
&\quad+ \int_{ F}f\,dx- \int_{E}f\,dx.\nonumber
\end{align}
The last two terms can be estimated by $ \int_{\Omega \cap B_\rho }f\,dx \leq c(n) \| f\|_{L^\infty}\rho^n$, cf. \cite{Massari:1974}. In remains to estimate the difference of the nonlocal terms.
Setting $A:=\Omega \cap B_R(x_0)$, we estimate for any $p>n$
\begin{align*}
& \int_F\int_FG(x,y)\,dx\,dy  - \int_E\int_E G(x,y)\,dx\,dy \\
&\quad \leq \int_F\left(\int_{F\cap B_\rho} G(x,y)\,dx - \int_{E\cap B_\rho}G(x,y)\,dx \right)\,dy+\int_{E\Delta F}|\phi_E|\,dx \\
&\quad \leq \int_{E\cup (\Omega \cap B_\rho)}\left(\int_{\Omega \cap B_\rho} |\chi_{F\cap B_\rho}(x)-\chi_{E\cap B_\rho}(x)||G(x,y)|\,dx \right)\,dy+\int_{B_\rho}|\phi_E|\,dx \\
&\quad \leq \|G\|_{L^1(A \times A) } |B_\rho| +\left( \int_{A}|\phi_E|^p\,dx\right)^\frac{1}{p} |B_\rho|^{1-\frac{1}{p}}\\
&\quad \leq c(n,p,G,E) \rho^{n-1+(1-\frac{n}{p})}.
\end{align*}
The claim follows with $\varepsilon= 1-\frac{n}{p}$ for any $p>n$ and $K=c(n,p,G,E)$ (or $K=c(n,p,G,E,r)$ in case of a volume constraint with $r$ as in Remark~\ref{intextpoints}).
\end{proof}

Theorem~\ref{thm:regularity} and Proposition~\ref{prop:quasimin} in conjunction with the results of Gr\"uter \cite{Grueter:1987} in which boundary regularity of $(K,\varepsilon)$-minimizers with weakly orthogonal surface measure was shown, immediately imply the following
\begin{theorem}[cf. {\cite[Theorem 3.2]{Julin-Pisante:2013}}]
Let $E\in \mathcal A$ or $\mathcal A_m$ be a local minimizer of $\mathcal E_\gamma$ in $\mathcal A$ or $\mathcal A_m$. Then
\begin{enumerate}
\item ${\rm reg}(\mu_E)$ is of class $C^{1,\alpha}$ for all $\alpha \in (0,1)$, ${\rm reg}(\mu_E)\cap\Omega$ is of class $C^{3,\alpha}$ for all $\alpha \in (0,1)$ and has mean curvature $ H = \lambda- 2\gamma \phi_E -f$ for some constant $\lambda \in \mathbb R$. If $x\in {\rm reg}(\mu_E)\cap \partial \Omega$ then ${\rm reg}(\mu_E)$ and $\partial \Omega$ intersect orthogonally in a neighborhood of $x$.
\item $\mathcal H^s({\rm sing}(\mu_E) ) = 0\quad\text{for all $s>n-8$}.$
\end{enumerate}
\end{theorem}
\begin{definition}[cf. \cite{Grueter:1987}] 
Here ${\rm reg}(\mu_E)$ is defined as the set of all points in $\partial E \cap \overline \Omega ={\rm spt}(\mu_E)$ such that one of the following alternatives holds.
\begin{enumerate}
\item If $x\in {\rm reg}(\mu_E) \cap \Omega$ there exits an oriented $C^1$-hypersurface $M_x$ such that $\mu_E = \mathcal H^{n-1}\llcorner M_x$ and $\nu_E= \nu_{M_x}$ in a neighborhood of $x$.
\item If $x\in {\rm reg}(\mu_E) \cap \partial\Omega$ there exits an oriented $C^1$-hypersurface $M_x'$ with boundary inside $\partial \Omega$ such that $\mu_E = \mathcal H^{n-1}\llcorner M_x'$ and $\nu_E= \nu_{M_x'}$ in a neighborhood of $x$.
\end{enumerate}
And ${\rm sing}(\mu_E):= \partial E \cap \overline \Omega\setminus {\rm reg}(\mu_E)$.
\end{definition}

\begin{remark}
In case $\Omega$ is of class $C^{k,\alpha}$ for $k=2,3$ we get that ${\rm reg}(\mu_E)$ is of class $C^{k,\alpha}$ (up to the boundary).
\end{remark}
\noindent \textbf{Acknowledgments} The research of the first-named author was
supported by the Herchel Smith Research Fellowship at the University of Cambridge and NSF grant  DMS-0807347. The first-named author
would like to thank Theodora Bourni and Robert Haslhofer
for helpful discussions throughout the course of this work.

\bibliographystyle{amsplain}
\bibliography{references}

\end{document}